\documentclass[11pt,a4paper]{amsart}
\newif\ifprivate
\usepackage{a4wide}
\usepackage{graphicx}
\usepackage{subfigure}
\usepackage{ifthen}

\allowdisplaybreaks

\newtheorem{theorem}{Theorem}
\newtheorem{lemma}{Lemma}[section]
\newtheorem{proposition}[lemma]{Proposition}
\theoremstyle{definition}
\newtheorem{definition}[lemma]{Definition}

\renewcommand{\MR}[1]{}

\newif\ifpdf
\ifx\pdfoutput\undefined
    \pdffalse          
\else
\ifthenelse{\equal{\the\pdfoutput}{1}}{\pdftrue}{\pdffalse}
\fi

\ifpdf
\fi

\begin{document}

\title{Chemical Trees Minimizing Energy and Hosoya Index}

\author{Clemens Heuberger}
\thanks{This paper was written while C.~Heuberger 
  was a visitor at the
  Center of Experimental Mathematics at the University of Stellenbosch. He
  thanks the center
  for its hospitality. He was also supported by the Austrian
  Science Foundation FWF, projects S9606, that is part of the
  Austrian National Research Network ``Analytic Combinatorics
  and Probabilistic Number Theory.''}
\address{Institut f\"ur Mathematik B\\Technische Universit\"at Graz\\Austria}
\email{clemens.heuberger@tugraz.at}
\author{Stephan G. Wagner}
\address{Department of Mathematical Sciences\\University of Stellenbosch\\South Africa}
\email{swagner@sun.ac.za}


\begin{abstract}
The energy of a molecular graph is a popular parameter that is defined as the sum of the absolute values of a graph's eigenvalues. It is well known that the energy is related to the matching polynomial and thus also to the Hosoya index via a certain Coulson integral. Trees minimizing the energy under various additional conditions have been determined in the past, e.g., trees with a given diameter or trees with a perfect matching. However, it is quite a natural problem to minimize the energy of trees with bounded maximum degree---clearly, the case of maximum degree 4 (so-called chemical trees) is the most important one. We will show that the trees with given maximum degree that minimize the energy are the same that have been shown previously to minimize the Hosoya index and maximize the Merrifield-Simmons index, thus also proving a conjecture due to Fischermann et al. Finally, we show that the minimal energy grows linearly with the size of the trees, with explicitly computable growth constants that only depend on the maximum degree.
\end{abstract}
\date{\today}

\maketitle
\ifprivate \thispagestyle{headings}\pagestyle{headings} \markboth{\jobname{}
rev. \SVNRevision{} ---
  \SVNDate{} \SVNTime}{\jobname{} rev. \SVNRevision{} --- \SVNDate{} \SVNTime}
\fi

\section{Introduction and statement of results}

The energy is a graph parameter stemming from the H\"uckel molecular orbital (HMO) approximation for the total $\pi$-electron energy, see \cite{Gutman:1978:energy,Gutman:2001:energy}. It is defined as the sum of the absolute values of all eigenvalues of a graph: if $\lambda_1, \lambda_2, \ldots, \lambda_n$ denotes the spectrum of a graph $G$ (i.e. the spectrum of its adjacency matrix), one has
$$E(G) = \sum_{i=1}^n |\lambda_i|.$$
For more information about the chemical importance of the energy as well as important properties, we refer to the book \cite{Gutman:1986:mathematical} and the survey article \cite{Gutman:2001:energy}. 
It is known that there are many interesting relations between the spectrum and matchings in the case of trees. This is due to the fact that the characteristic polynomial $\phi(T,x) = \det (x I - A(T))$, where $A(T)$ denotes the adjacency matrix of a tree $T$ and $I$ the identity matrix, can be expressed as \cite{Cvetkovic:1995:spectra}\begin{equation}\label{eq:characteristic-matching}
\phi(T,x) = \sum_{k \geq 0} (-1)^k m(T,k) x^{n-2k},
\end{equation}
where $m(T,k)$ denotes the number of matchings of $T$ of cardinality $k$. It follows from this representation that the energy can actually be computed by means of the Coulson integral~\cite{Gutman:1986:mathematical}
\begin{equation}
\label{eq:coul}
  E(T)=\frac2{\pi} \int_{0}^{\infty} x^{-2} \log\left(\sum_{k}m(T,k)x^{2k}\right)\,dx.
\end{equation}
This connection to matchings has been used in various instances to determine the extremal values of the energy within certain classes of graphs (trees or graphs that are closely related to trees, such as unicyclic graphs---see \cite{Hou:2001:unicyclic}). We mention, for instance, the recent paper of Yan and Ye \cite{Yan:2005:minimal}, where trees with prescribed diameter maximizing the energy are characterized. An earlier example is the article by Zhang and Li \cite{Zhang:1999:acyclic}, where trees with a perfect matching are studied.

For obvious reasons, it is also a natural problem to study \emph{chemical trees} (i.e., trees with maximum degree $\leq 4$) and generally trees with bounded maximum degree. Fischermann et al. \cite{Fischermann:2002:extremal} noticed that the chemical trees minimizing the energy agree with those minimizing the Hosoya index (i.e., the total number of matchings, see \cite{Hosoya:1986:topological}) for a small number of vertices. Indeed, we will show that this is always the case and also holds for arbitrary given maximum degree. The resulting trees have also been shown to maximize the Merrifield-Simmons index (i.e., the total number of independent vertex subsets, see \cite{Merrifield:1989:topological}) in an earlier paper by the authors \cite{Heuberger:2008:maximizing}. We will show that essentially the same method can be used again. In view of the representation~\eqref{eq:coul}, it is sufficient to minimize the polynomial 
\begin{equation*}
  M(T,x)=\sum_k m(T,k)x^{k}
\end{equation*}
for all positive values of $x$. Surprisingly, it turns out that the result of this minimization problem does not depend on
$x$. In order to state our result, we use the notion of \emph{complete $d$-ary trees}: the complete $d$-ary tree of height $h-1$ is denoted by $C_h$, i.e., $C_1$ is a single vertex and $C_{h}$ has $d$ branches $C_{h-1}$, \ldots,
$C_{h-1}$, cf.\ Figure~\ref{fig:complete-binary-trees}. It is convenient to
set $C_0$ to be the empty graph. 
\begin{figure}[htbp]
  \centering
  \subfigure[$C_1$ for all $d$]{\hbox to 0.2\linewidth{\hfill\includegraphics{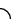}\hfill}}\hfil
  \subfigure[$C_2$ for $d=2$]{\hbox to 0.2\linewidth{\hfill\includegraphics{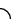}\hfill}}\hfil
  \subfigure[$C_2$ for $d=3$]{\hbox to 0.2\linewidth{\hfill\includegraphics{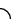}\hfill}}\hfil
  \subfigure[$C_3$ for $d=2$]{\hbox to 0.2\linewidth{\hfill\includegraphics{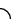}\hfill}}
  \caption{Complete $d$-ary trees}
  \label{fig:complete-binary-trees}
\end{figure}

Let $\mathcal{T}_{n,d}$ be the set of all trees with $n$ vertices and maximum degree $\le d+1$. We define a special tree $T^{*}_{n,d}$ as follows (see also \cite{Heuberger:2008:maximizing}):
\begin{definition}\label{definition:wagner-tree}
$T^{*}_{n,d}$ is the tree with $n$ vertices that can be decomposed as
  \begin{center}
    \includegraphics{energy_35.eps}
  \end{center}
  with $B_{k,1}$, \ldots, $B_{k,d-1}\in\{C_k,C_{k+2}\}$ for $0\le k<\ell$ and either   $B_{\ell,1}=\cdots=B_{\ell,d}=C_{\ell-1}$ or $B_{\ell,1}=\cdots=B_{\ell,d}=C_\ell$ or 
  $B_{\ell,1}$, \ldots, $B_{\ell,d}\in\{C_{\ell},C_{\ell+1},C_{\ell+2}\}$,
  where at least two of $B_{\ell,1}$, \ldots, $B_{\ell,d}$ equal $C_{\ell+1}$. This representation is unique, and one has the ``digital expansion''
  \begin{equation}\label{eq:digital-expansion}
  (d-1)n + 1 = \sum_{k=0}^{\ell} a_k d^k,
  \end{equation}
  where $a_k = (d-1)(1 + (d+1)r_k)$ and $0\le r_k\le d-1$ is the number of $B_{k,i}$ that are isomorphic to $C_{k+2}$ for $k < \ell$, and
  \begin{itemize}
  \item $a_{\ell} = 1$ if $B_{\ell,1}=\cdots=B_{\ell,d}=C_{\ell-1}$,
  \item $a_{\ell} = d$ if $B_{\ell,1}=\cdots=B_{\ell,d}=C_{\ell}$,
  \item or otherwise $a_\ell = d + (d-1)q_\ell + (d^2-1)r_\ell$, where $q_\ell \geq 2$ is the number of $B_{\ell,i}$ that are isomorphic to $C_{\ell+1}$ and $r_\ell$ the number of $B_{\ell,i}$ that are isomorphic to $C_{\ell+2}$.
  \end{itemize}
\end{definition}

The tree $T^{*}_{n,d}$ has already been shown to minimize the Hosoya index and maximize
the Merrifield-Simmons index over $\mathcal{T}_{n,d}$ \cite{Heuberger:2008:maximizing}. In the present paper, we will prove the following result:

\begin{proposition}\label{prop:minimal-tree-x}
  Let $n$ and $d$ be positive integers and $x>0$. Then $T^{*}_{n,d}$ is the unique (up to isomorphism) tree in $\mathcal{T}_{n,d}$ that minimizes $M(T,x)$.
\end{proposition}

Note that the Hosoya index is exactly $M(T,1)$, and so it is a trivial corollary that $T^{*}_{n,d}$ minimizes the Hosoya index. Our main theorem is another immediate consequence that follows from~\eqref{eq:coul}:

\begin{theorem}\label{main-th}
  Let $n$ and $d$ be positive integers. Then $T^{*}_{n,d}$ is the unique (up to isomorphism)
  tree in $\mathcal{T}_{n,d}$ that minimizes the energy.
\end{theorem}

In our final section, we will study the asymptotic behavior of the minimal energy:

\begin{theorem}\label{th:asympt}
The energy of $T^{*}_{n,d}$ is asymptotically
$$E(T^{*}_{n,d}) = \alpha_d n + O(\log n),$$
where
\begin{equation}\label{eq:alpha-d-definition}
\alpha_d = 2\sqrt{d}(d-1)^2 \left( \sum_{\substack{j\geq 1\\j \equiv 0 \mod
      2}} d^{-j} \left( \cot \frac{\pi}{2j} - 1 \right) + \sum_{\substack{j
      \geq 1\\j \equiv 1 \mod 2}} d^{-j} \left( \csc \frac{\pi}{2j} - 1 \right)
\right)
\end{equation}
is a constant that only depends on $d$.
\end{theorem}

\section{Recursive Formul\ae{} for Rooted Trees}

Let $x>0$ be fixed.
In order to derive recursive formul\ae{} for $M(T,x)$, we fix a root $r$ of $T$ and 
we define $m_1(T,k)$ to be the number of matchings of cardinality
$k$ covering the root and $m_0(T,k)$ to be the number of matchings of
cardinality $k$ not covering the root.

We write $M_j(T,x)=\sum_{k} m_j(T,k)x^k$ for
$j\in\{0,1\}$. Obviously, we have $M(T,x)=M_0(T,x)+M_1(T,x)$. The ratio 
\begin{equation}\label{eq:tau-definition}
  \tau(T,x)=\frac{M_0(T,x)}{M(T,x)}
\end{equation}
will be an important auxiliary quantity in our proofs. The following lemma summarizes important properties of all these quantities.

\begin{lemma}\label{le:recursion}
  Let $T_1$,
  \ldots, $T_{\ell}$ be the branches of the rooted tree $T$. Then the following
  recursive formul\ae{} hold:
  \begin{align}
    M_0(T,x)&=\prod_{i=1}^{\ell} M(T_i,x),\label{eq:recursion-M-0}\\
    M_1(T,x)&=x\sum_{i=1}^{\ell} M_0(T_i,x) \prod_{\substack{j=1\\j\neq
        i}}^{\ell}M(T_j,x),\label{eq:recursion-M-1}\\
    \tau(T,x)&=\frac{1}{1+x\sum_{i=1}^{\ell} \tau(T_i,x)}.\label{eq:recursion-tau}
  \end{align}
\end{lemma}
\begin{proof}
  A matching not covering the root corresponds to a selection of an arbitrary
  matching in each branch. If the matching of $T$ is required to have
  cardinality $k$, the cardinalities of the corresponding matchings in the
  branches must add up to $k$. This is exactly the coefficient of $x^{k}$ in
  the product $\prod_{i=1}^{\ell}M(T_i,x)$. This proves \eqref{eq:recursion-M-0}.

  Every matching of $T$ covering the root contains exactly one edge between
  the root and some branch $T_i$. In this branch, a matching not covering the root of $T_i$ may be
  chosen, in all other branches, arbitrary matchings are allowed. The
  cardinality of the matching of $T$ is then the sum of the cardinalities of the
  matchings in the branches plus one for the edge incident to the root.
  This yields \eqref{eq:recursion-M-1}.

  Finally, \eqref{eq:recursion-tau} is an immediate consequence of
  \eqref{eq:tau-definition}, \eqref{eq:recursion-M-0} and \eqref{eq:recursion-M-1}.
\end{proof}

In the following, we will fix a positive integer $d$ and consider only trees whose maximum degree is $\leq d+1$. First of all, we study the behavior of the sequence $\tau(C_h,x)$. It is convenient to set $M_0(C_0,x)=0$ and $M_1(C_0,x)=1$ for the polynomials associated to the empty tree. Note that this choice allows adding empty branches without disturbing the recursive formul\ae{} \eqref{eq:recursion-M-0},
\eqref{eq:recursion-M-1}, \eqref{eq:recursion-tau}. Then
\eqref{eq:recursion-tau} translates into a recursion for $\tau(C_h,k)$ as
follows:
\begin{equation*}
  \tau(C_0,x)=0,\quad \tau(C_1,x)=1, \quad \tau(C_h,x)=\frac{1}{1+dx\tau(C_{h-1},x)}.
\end{equation*}
It is an easy exercise to prove the following explicit formula for $\tau(C_h,x)$ by means of induction:
\begin{lemma}\label{le:tau-complete-exact}
For every $x > 0$, we have
$$\tau(C_h,x) = \frac{\left( \frac{1+\sqrt{1+4dx}}{2} \right)^h - \left( \frac{1-\sqrt{1+4dx}}{2} \right)^h}{\left( \frac{1+\sqrt{1+4dx}}{2} \right)^{h+1} - \left( \frac{1-\sqrt{1+4dx}}{2} \right)^{h+1}}.$$
\end{lemma}
Now, the limit behavior of $\tau(C_h,x)$ for positive $x$ follows immediately.

\begin{lemma}\label{le:tau-complete-convergence}
  For every $x>0$, the subsequence  $\tau(C_{2h},x)$ is strictly increasing, whereas the
  subsequence $\tau(C_{2h+1},x)$ is strictly decreasing. Both subsequences are
  converging to the same limit $\frac{2}{1+\sqrt{1+4dx}}$ and we have
  \begin{equation*}
    0=\tau(C_0,x)<\tau(C_2,x)<\cdots< \frac{2}{1+\sqrt{1+4dx}} <\cdots <\tau(C_3,x)<\tau(C_1,x)=1.
  \end{equation*}
\end{lemma}

From this point on, the method of \cite{Heuberger:2008:maximizing}, in particular
Section~5, can be applied. It suffices to replace $d$ by $dx$ at certain points
of the proof. For the sake of completeness, we state the relevant auxiliary results.

\begin{lemma}\label{le:trivial-tau-estimate}
  Let $T$ be a rooted tree and $x>0$. Then 
  \begin{equation*}
    \frac{1}{dx+1}\le \tau(T,x)\le 1,
  \end{equation*}
  unless $T$ is empty, where $\tau(T,x)=0$.
\end{lemma}
\begin{proof}
  To prove this lemma, use induction and Lemma~\ref{le:recursion}.
\end{proof}

\begin{definition}\label{def:outline-tree}
  Let $T$ be a possibly rooted tree. Then we construct the \emph{outline graph} of $T$ by replacing all maximal subtrees isomorphic to some $C_k$, $k\ge 0$, by a special leaf $C_k$.  In this process, we attach $(d+1-r)$ leaves $C_0$ to internal nodes (non-leaves and non-root) of degree $r$ with $2\le r \le d$. If $T$ is a rooted tree with a root of degree $r$ ($1\le r\le d$), then we also attach $d-r$ leaves $C_0$ to it.
\end{definition}

The construction ensures that the outline graph of a rooted tree is a rooted $d$-ary
tree, and that the outline graph of an arbitrary tree of maximum degree $\leq d+1$ has only vertices of degree $1$ and $d+1$. An example is shown in Figure~\ref{fig:reduction-example}. The outline of a rooted tree $C_k$ is just the rooted tree consisting of the single leaf $C_k$.
\begin{figure}[htbp]
  \centering
  \includegraphics{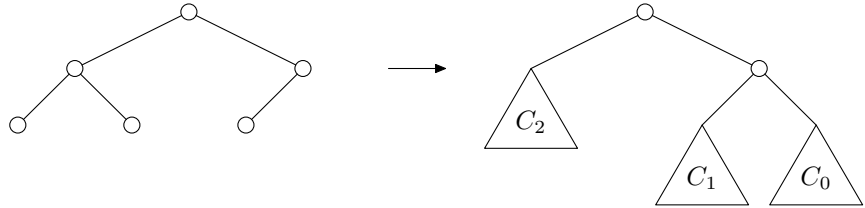}
  \caption{Reduction to the outline graph ($d=2$)}
  \label{fig:reduction-example}
\end{figure}

If enough information on the outline of a rooted tree is available, we can
determine it from its $\tau(T,x)$-value.

\begin{lemma}\label{le:tau-deconstruction}
  Let $j\ge 0$ be an integer and $T$ be a rooted tree whose outline does not
  contain any $C_k$ for $0\le k\le j-3$ and $x>0$. If
  \begin{equation*}
    j \text{ is odd and }\tau(C_{j},x)\le \tau(T,x)
  \end{equation*}
  or 
  \begin{equation*}
    j \text{ is even and } \tau(T,x)\le\tau(C_{j},x),
  \end{equation*}
  then $T\in\{C_{j-2}, C_{j}\}$.
\end{lemma}
\begin{proof}
  The inductive proof is analogous to that of
  \cite[Lemma~3.4]{Heuberger:2008:maximizing}, the only difference being
  the fact that here, sums are considered instead of products.
\end{proof}

\section{Minimal Trees with Respect to $x$}

We say that $T$ is a minimal tree with respect to some $x>0$, if it minimizes  $M(T,x)$ among all trees in $\mathcal{T}_{n,d}$.

The key lemma is an exchange lemma which gives a local optimality criterion.
\begin{lemma}\label{le:edge-exchange}
  Let $x>0$ and let $T$ be a minimal tree with respect to $x$. If there are
  (possibly empty) rooted trees $L_1$, \ldots $L_d$, $R_1$, \ldots, $R_d$ and a
  tree $T_0$ such that $T$ can be decomposed as
  \begin{center}
    \includegraphics{ms-tree-3_6.eps}
  \end{center}
  and such that $\tau(L_1,x)<\tau(R_1,x)$ (after appropriate reordering of the
  $L_i$'s and the $R_i$'s), then
  \begin{equation*}
    \max\{\tau(L_i,x): 1\le i\le d \} \le \min\{\tau(R_i,x): 1\le i\le d\}.
  \end{equation*}
  \end{lemma}
\begin{proof}
  We need four auxiliary quantities:
  \begin{itemize}
  \item $m_{00}(T_0,k)$: number of matchings of $T_0$ of cardinality $k$ where neither $v$ nor $w$ is covered.
  \item $m_{10}(T_0,k)$: number of matchings of $T_0$ of cardinality $k$ where $v$ is covered, but $w$ is not.
  \item $m_{01}(T_0,k)$: number of matchings of $T_0$ of cardinality $k$ where $w$ is covered, but $v$ is not.
  \item $m_{11}(T_0,k)$: number of matchings of $T_0$ of cardinality $k$ where both $v$ and $w$ are covered.
  \end{itemize}
  The corresponding polynomials are denoted by $M_{ij}(T_0,x)=\sum_k m_{ij}(T_0,k)x^k$. 
  Define 
  \begin{multline*}
    G(L_1,\ldots, L_d,R_1,\ldots,
    R_d;x):=M_{00}(T_0,x)\bigg(1+x\sum_{i=1}^d \tau(L_i,x)\bigg)\bigg(1+x\sum_{i=1}^d \tau(R_i,x)\bigg)\\
    +M_{10}(T_0,x)\bigg(1+x\sum_{i=1}^d \tau(R_i,x)\bigg)+M_{01}(T_0,x)\bigg(1+x\sum_{i=1}^d \tau(L_i,x)\bigg)+M_{11}(T_0,x).
  \end{multline*}

  Then it is easily seen that
  \begin{equation*}
    M(T,x)=G(L_1,\ldots, L_d,R_1,\ldots,R_d;x)\prod_{i=1}^d M(L_i,x) \prod_{i=1}^d M(R_i,x).
  \end{equation*}
  In view of the minimality of $M(T,x)$, we must have
  \begin{equation*}
    G(L_1,\ldots,L_d,R_1,\ldots,R_d;x)\le G(\pi(L_1),\ldots,\pi(L_d),\pi(R_1),\ldots,\pi(R_d);x)
  \end{equation*}
  for all permutations $\pi$ of $\{L_1,\ldots,L_d,R_1,\ldots,R_d\}$. Ignoring the assumption
  $\tau(L_1,x)<\tau(R_1,x)$ for the moment, we see that the minimum of the first summand among all possible permutations is attained if
  \begin{multline}\label{eq:rearrangement-possibilities-tau}
    \max\{\tau(L_i,x): i=1,\ldots,d \} \le \min\{\tau(R_i,x): i=1,\ldots,d\} \text{ or } \\ \min\{\tau(L_i,x): i=1,\ldots,d \} \ge \max\{\tau(R_i,x): i=1,\ldots,d\}
  \end{multline}
  by standard arguments (note that the sum of the two factors does not depend
  on the permutation). The sum of the second and the third summand is
  minimized if $\max\{\tau(L_i,x): i=1,\ldots,d \} \le \min\{\tau(R_i,x): i=1,\ldots,d\}$
  in the case $M_{10}(T_0,x) \leq M_{01}(T_0,x)$ and is minimized if $\min\{\tau(L_i,x): i=1,\ldots,d \} \ge \max\{\tau(R_i,x): i=1,\ldots,d\}$  in the case that $M_{10}(T_0,x) \geq
  M_{01}(T_0,x)$. Therefore, the minimality of $G$ yields
  \eqref{eq:rearrangement-possibilities-tau}. The assumption
  $\tau(L_1,x)<\tau(R_1,x)$ implies the first possibility.
\end{proof}

The remaining steps of the proof in \cite{Heuberger:2008:maximizing} do
not depend on the particular recursions any more, they only depend on the
monotonicity properties in Lemma~\ref{le:tau-complete-convergence} and
Lemma~\ref{le:tau-deconstruction} as well as the local optimality criterion in
Lemma~\ref{le:edge-exchange}. The only difference is that all inequalities have
to reversed (and $\rho(T)$ has to be replaced by $\tau(T,x)$). Thus the
basis of the induction can be formulated as follows.

\begin{lemma}
  Let $T$ be a minimal tree with respect to $x>0$ and let $j$ be the least nonnegative integer such that
  the outline graph of $T$ contains a $C_j$. Then the outline graph of $T$
  contains $C_j$ at most $(d-1)$ times and there is a vertex $v$ of the outline
  graph of $T$ which is adjacent to all copies of $C_j$ in the outline graph of $T$.
\end{lemma}
\begin{proof}
  Analogous to the proof of \cite[Lemma~4.3]{Heuberger:2008:maximizing}.
\end{proof}

The inductive step can be formulated as follows.

\begin{lemma}\label{le:inductive-step}
  Let $x>0$, $T$ be a minimal tree with respect to $x$, $k$ be a nonnegative integer and assume that the
  outline graph of $T$ can be decomposed as
  \begin{center}
    \includegraphics{ms-tree-3_9.eps}
  \end{center}
  for some rooted trees $L_k$ (possibly empty) and $R_k$ with 
  \begin{subequations}
    \begin{equation*}
      k\text{ is even and }\tau(C_k,x)< \tau(L_k,x)<\tau(C_{k+2},x)
    \end{equation*}
    or
    \begin{equation*}
      k\text{ is odd and }\tau(C_{k+2},x)< \tau(L_k,x)<\tau(C_{k},x)
    \end{equation*}
  \end{subequations}
  or
  \begin{equation*}
    L_k=C_k.
  \end{equation*}
  Assume that $R_k$ is non-empty and the outline of $R_k$ does not contain any $C_{\ell}$ with $\ell<k$.

  Then exactly one of the following assertions is true:
  \begin{enumerate}
  \item $R_k\in \{C_{k}, C_{k+1}, C_{k+3} \}$,
  \item $R_k$ consists of $d$ branches
    $C_{k+1}$, $C_{k+1}$, $C_{\ell_3}$, \ldots,
    $C_{\ell_d}$ with $\ell_i\in\{k,k+1,k+2\}$ for $3\le i\le d$,
  \item the outline of $R_k$ can be decomposed as
    \begin{center}
      \includegraphics{energy_30.eps}
    \end{center}
    for $B_{k,1}$, \ldots, $B_{k,d-1}\in\{C_k,C_{k+2}\}$ and a non-empty rooted tree $R_{k+1}$ whose outline does
    not contain any $C_{\ell}$ for $\ell\le k$. Furthermore,
    \begin{subequations}
      \begin{equation*}
        k\text{ is even and }\tau(C_{k+3},x)< \tau(L_{k+1},x)<\tau(C_{k+1},x)
      \end{equation*}
      or
      \begin{equation*}
        k\text{ is odd and }\tau(C_{k+1},x)< \tau(L_{k+1},x)<\tau(C_{k+3},x)
      \end{equation*}
    \end{subequations}
    where $L_{k+1}$ is defined as follows:
    \begin{center}
      \includegraphics{energy_31.eps}
    \end{center}
  \end{enumerate}
\end{lemma}
\begin{proof}
  Analogous to the proof of \cite[Lemma~4.4]{Heuberger:2008:maximizing}.
\end{proof}

Repeated application of Lemma~\ref{le:inductive-step} now yields Proposition~\ref{prop:minimal-tree-x} and thus also our main theorem.

\begin{proof}[Proof of Proposition~\ref{prop:minimal-tree-x}]
  Analogous to the proof of \cite[Theorem~1]{Heuberger:2008:maximizing}.
\end{proof}

\section{The value of the minimal energy}

Since the extremal trees are described in terms of complete $d$-ary trees, we have to study the energy of these trees first. Note that, in view of \eqref{eq:recursion-M-0} and the definition of $\tau(M,x)$, we have
$$M(C_h,x) = \frac{1}{\tau(C_h,x)} \cdot M(C_{h-1},x)^d.$$
Iterating this equation yields
$$M(C_h,x) = \prod_{j=1}^h \tau(C_j,x)^{-d^{h-j}},$$
and in view of Lemma~\ref{le:tau-complete-exact} this gives us the explicit
formula
$$M(C_h,x) = \prod_{j=1}^h \left(\frac{Q_{j+1}(x)}{Q_j(x)} \right)^{d^{h-j}},$$
where 
\begin{align*}
  Q_j(x)&:= \frac{u(x)^j - v(x)^j}{u(x)-v(x)},\\
  u(x) &:= \frac{1+\sqrt{1+4dx}}{2},\\
  v(x) &:= \frac{1-\sqrt{1+4dx}}{2}.
\end{align*}
The denominator $u(x)-v(x)$ has been introduced such that $Q_j(x)$ is always a
polynomial: Indeed, the recursion
$$Q_1(x) \equiv 1,\ Q_2(x) \equiv 1,\quad Q_j(x) = Q_{j-1}(x) + dx Q_{j-2}(x)$$
holds, and it follows by induction that $Q_j$ is a polynomial of degree $\lfloor (j-1)/2 \rfloor$.

Now we have
\begin{align*}
  M(C_h,x)&=Q_{h+1}(x) Q_1(x)^{-d^{h-1}}\prod_{j=2}^h Q_j(x)^{d^{h+1-j}}\prod_{j=2}^h
  Q_j(x)^{-d^{h-j}}\\
  &= Q_{h+1}(x) \prod_{j=1}^h Q_j(x)^{(d-1)d^{h-j}},
\end{align*}
where the fact that $Q_1(x)=1$ has been used. It turns out that the zeros of $Q_j$ can be explicitly computed. If $Q_j(x) = 0$, then $u(x)^j = v(x)^j$, and so
$$\frac{u(x)}{v(x)} = \frac{1+\sqrt{1+4dx}}{1-\sqrt{1+4dx}}$$
has to be an $j$-th root of unity $\zeta$. Then,
$$x = - \frac{\zeta}{d(1+\zeta)^2} = - \frac{1}{2d(1+\operatorname{Re}(\zeta))}.$$
Thus, $x$ has to be of the form
$$x = - \frac{1}{2d \left(1 + \cos \frac{2k\pi}{j} \right)}$$
for some $0 \leq k < \frac{j}{2}$. However, note that $x = - \frac{1}{4d}$ is also a zero of the denominator $u(x)-v(x) = \sqrt{1+4dx}$, and that there are no double zeros, since the derivative is given by
$$\frac{d}{dx} \left(u(x)^j - v(x)^j\right) = \frac{jd}{\sqrt{1+4dx}} \left( u(x)^{j-1} + v(x)^{j-1} \right),$$
which cannot be $0$ if $u(x) = \zeta v(x)$ for a $j$-th root of unity $\zeta \neq -1$. Hence, the zeros of $Q_j$ are precisely the numbers
$$- \frac{1}{2d \left(1 + \cos \frac{2k\pi}{j} \right)},\quad k = 1,2,\ldots,\left\lfloor \frac{j-1}{2} \right\rfloor,$$
and it follows from~\eqref{eq:characteristic-matching} that the characteristic polynomial $\phi(C_h,x)$ can be written as
$$\phi(C_h,x) = x^n M(C_h,-x^{-2}) = x^n Q_{h+1}(-x^{-2}) \prod_{j=1}^h Q_j(-x^{-2})^{(d-1)d^{h-j}}.$$
Hence, the nonzero eigenvalues of $\phi(C_h,x)$ are
$$\pm \sqrt{2d \left(1 + \cos \frac{2k\pi}{j} \right)} = \pm 2\sqrt{d} \cos \frac{k\pi}{j},\quad k = 1,2,\ldots,\left\lfloor \frac{j-1}{2} \right\rfloor,$$
with multiplicity $(d-1)d^{h-j}$ for $j=1,2,\ldots,h$ and multiplicity 1 for $j=h+1$. If follows that the energy of $C_h$ is given by
$$E(C_h) = \left( \sum_{j=1}^h (d-1)d^{h-j} \sum_{k=1}^{\lfloor (j-1)/2 \rfloor} 4\sqrt{d} \cos \frac{k\pi}{j} \right) + \sum_{k=1}^{\lfloor h/2 \rfloor} 4\sqrt{d} \cos \frac{k\pi}{h+1}.$$
Noticing that
$$\sum_{k=1}^{\lfloor (j-1)/2 \rfloor} \cos \frac{k\pi}{j} = \begin{cases} \frac{1}{2} \left( \cot \frac{\pi}{2j} - 1 \right) & j \equiv 0 \mod 2, \\ \frac{1}{2} \left( \csc \frac{\pi}{2j} - 1 \right) & j \equiv 1 \mod 2, \end{cases}$$
this reduces to
\begin{align*}
E(C_h) &= 2\sqrt{d}(d-1) \left( \sum_{\substack{j=1\\j \equiv 0 \mod 2}}^h d^{h-j} \left( \cot \frac{\pi}{2j} - 1 \right) + \sum_{\substack{j=1\\j \equiv 1 \mod 2}}^h d^{h-j} \left( \csc \frac{\pi}{2j} - 1 \right) \right) \\
&\ \ \ + \begin{cases} 2\sqrt{d} \left( \csc \frac{\pi}{2(h+1)} - 1 \right) & h \equiv 0 \mod 2, \\ 2\sqrt{d} \left( \cot \frac{\pi}{2(h+1)} - 1 \right) & h \equiv 1 \mod 2. \end{cases}
\end{align*}
Next, we determine the asymptotic behavior of the energy of $C_h$:
\begin{lemma}
The energy of a complete $d$-ary tree $C_h$ satisfies
$$E(C_h) = \alpha_d |C_h| + O(1),$$
where $|C_h|$ denotes the number of vertices of $C_h$ and $\alpha_d$ is given
by \eqref{eq:alpha-d-definition}.
\end{lemma}
\begin{proof}
Note that $|C_h| = \frac{d^h-1}{d-1} = \frac{d^h}{d-1} + O(1)$ and $\cot \frac{\pi}{2j} = \frac{2j}{\pi} + O(1)$, $\csc \frac{\pi}{2j} = \frac{2j}{\pi} + O(1)$, so that
\begin{align*}
E(C_h) &= 2\sqrt{d}(d-1)d^h \left( \sum_{\substack{j=1\\j \equiv 0 \mod 2}}^h d^{-j} \left( \cot \frac{\pi}{2j} - 1 \right) + \sum_{\substack{j=1\\j \equiv 1 \mod 2}}^h d^{-j} \left( \csc \frac{\pi}{2j} - 1 \right) \right)\\&\qquad + \frac{4\sqrt{d}h}{\pi} + O(1) \\
&= \alpha_d |C_h| - 2\sqrt{d}(d-1)d^h \left( \sum_{\substack{j > h\\j \equiv 0 \mod 2}} d^{-j} \left( \cot \frac{\pi}{2j} - 1 \right) + \sum_{\substack{j > h\\j \equiv 1 \mod 2}} d^{-j} \left( \csc \frac{\pi}{2j} - 1 \right) \right)\\&\qquad + \frac{4\sqrt{d}h}{\pi} + O(1) \\
&= \alpha_d |C_h| - 2\sqrt{d}(d-1)d^h \sum_{j > h} d^{-j} \left( \frac{2j}{\pi} + O(1) \right) + \frac{4\sqrt{d}h}{\pi} + O(1) \\
&= \alpha_d |C_h| - 2\sqrt{d}(d-1)d^h \left( \frac{2hd^{-h}}{(d-1)\pi} + O(d^{-h}) \right) + \frac{4\sqrt{d}h}{\pi} + O(1) \\
&= \alpha_d |C_h| + O(1),
\end{align*}
as claimed.
\end{proof}

Now, we are able to prove our main asymptotic result:

\begin{proof}[Proof of Theorem~\ref{th:asympt}]
Using the decomposition of  $T^*_{n,d}$  as shown in Definition~\ref{definition:wagner-tree}, we note that\begin{multline*}
  \left( \prod_{k=0}^{\ell-1} \prod_{j=1}^{d-1} M(B_{k,j},x) \right) \left(
    \prod_{j=1}^{d} M(B_{\ell,j},x) \right) \leq M(T^*_{n,d},x) \\
  \leq \left(
    \prod_{k=0}^{\ell-1} \prod_{j=1}^{d-1} M(B_{k,j},x) \right) \left(
    \prod_{j=1}^{d} M(B_{\ell,j},x) \right) (1+x)^{d(\ell+1)}
\end{multline*}
for arbitrary $x > 0$, since every matching in the union $\bigcup_k \bigcup_j B_{k,j}$ is also a matching in $T^*_{n,d}$, whereas every matching of $T^*_{n,d}$ consists of a matching in $\bigcup_k \bigcup_j B_{k,j}$ and a subset of the remaining $\leq d(\ell+1)$ edges. Making use of~\eqref{eq:coul} once again, this implies that
\begin{multline*}
  \sum_{k=0}^{\ell-1} \sum_{j=1}^{d-1} E(B_{k,j}) + \sum_{j=1}^{d}
  E(B_{\ell,j}) \leq E(T^*_{n,d}) \\ \leq \sum_{k=0}^{\ell-1} \sum_{j=1}^{d-1}
  E(B_{k,j}) + \sum_{j=1}^{d} E(B_{\ell,j}) + \frac2{\pi} d(\ell+1) \int_0^{\infty} x^{-2}
  \log(1+x^2)\,dx.
\end{multline*}
Since $\int_0^{\infty} x^{-2} \log(1+x^2)\,dx = \pi$, this implies that
\begin{align*}
E(T^*_{n,d}) &= \sum_{k=0}^{\ell-1} \sum_{j=1}^{d-1} E(B_{k,j}) + \sum_{j=1}^{d} E(B_{\ell,j}) + O(\ell) \\
&= \sum_{k=0}^{\ell-1} \sum_{j=1}^{d-1} \left( \alpha_d |B_{k,j}| + O(1) \right)
  + \sum_{j=1}^{d} \left( \alpha_d |B_{\ell,j}| + O(1) \right) + O(\ell) \\
  &= \alpha_d \left( \sum_{k=0}^{\ell-1} \sum_{j=1}^{d-1} |B_{k,j}| + \sum_{j=1}^{d} |B_{\ell,j}| \right) + O(\ell) \\
  &= \alpha_d (|T^*_{n,d}| - O(\ell)) + O(\ell) \\
  &= \alpha_d n + O(\ell).
\end{align*}
It is not difficult to see that $\ell=O(\log n)$ (this follows from~\eqref{eq:digital-expansion}, see \cite{Heuberger:2008:positional} for a detailed analysis), and so we finally have
$$E(T^*_{n,d}) = \alpha_d n + O(\log n),$$
which finishes the proof.
\end{proof}

The following table shows some numerical values of the constants $\alpha_d$:

\begin{center}
\begin{tabular}{|c|c|}
\hline
$d$ & $\alpha_d$ \\
\hline
2 & 1.102947505597 \\
3 & 0.970541979946 \\
4 & 0.874794345784 \\
5 & 0.802215758706 \\
6 & 0.744941364903 \\
7 & 0.698315075830 \\
8 & 0.659425329682 \\
9 & 0.626356806404 \\
10 & 0.597794680849 \\
20 & 0.434553264777 \\
50 & 0.279574397741 \\
100 & 0.198836515295 \\
\hline
\end{tabular}
\end{center}

\bibliographystyle{abbrv}
\bibliography{energy}

\end{document}